\documentclass[12pt,reqno]{amsart}

\usepackage{times}
\usepackage[T1]{fontenc}
\usepackage{mathrsfs}
\usepackage{latexsym}
\usepackage[dvips]{graphics}
\usepackage[dvips]{graphicx}
\usepackage{epsfig}
\usepackage{amsmath,amsfonts,amsthm,amssymb,amscd}
\input amssym.def
\input amssym.tex
\usepackage{color}
\usepackage{nicefrac}

\newcommand{\comment}[1]{}

\newcommand\be{\begin{equation}}
\newcommand\ee{\end{equation}}
\newcommand\bea{\begin{eqnarray}}
\newcommand\eea{\end{eqnarray}}
\newcommand\nbea{\begin{eqnarray*}}
\newcommand\neea{\end{eqnarray*}}
\newcommand\bi{\begin{itemize}}
\newcommand\ei{\end{itemize}}
\newcommand\ben{\begin{enumerate}}
\newcommand\een{\end{enumerate}}

\newtheorem{thm}{Theorem}[section]

\newtheorem{lem}[thm]{Lemma}

\newtheorem{rek}[thm]{Remark}

\newcommand{\Z}{\ensuremath{\mathbb{Z}}}

\numberwithin{equation}{section}

\newcommand{\bal}{\begin{align}}
\newcommand{\eal}{\end{align}}

\begin{document}

\title[Constructions of Generalized More Sums Than Differences Sets]{Explicit Constructions of Large Families of Generalized More Sums Than Differences Sets}

\author{Steven J. Miller, Luc Robinson and Sean Pegado}\email{sjm1@williams.edu (Steven.Miller.MC.96@aya.yale.edu), slr1@williams.edu, pegado.sean@gmail.com}
\address{Department of Mathematics and Statistics, Williams College,
Williamstown, MA 01267}

\subjclass[2010]{11P99 (primary), 11K99 (secondary)}

\keywords{sum-dominant sets, MSTD sets, $k$-generational sum-dominant sets}

\date{\today}

\thanks{We thank the participants of various CANT Conferences (especially Peter Hegarty, Mel Nathanson and Kevin O'Bryant) and the Number Theory and Probability Group of SMALL 2011 REU at Williams College for many enlightening conversations. The first and second named authors were partially supported by NSF grant DMS0970067; all three authors were partially supported by Williams College.}

\begin{abstract} A More Sums Than Differences (MSTD) set is a set of integers $A \subset \{0,\dots,n-1\}$ whose sumset $A+A$ is larger than its difference set $A-A$. While it is known that as $n\to\infty$ a positive percentage of subsets of $\{0,\dots,n-1\}$ are MSTD sets, the methods to prove this are probabilistic and do not yield nice, explicit constructions. Recently Miller, Orosz and Scheinerman \cite{MOS} gave explicit constructions of a large family of MSTD sets; though their density is less than a positive percentage, their family's density among subsets of $\{0,\dots,n-1\}$ is at least $C/n^4$ for some $C>0$, significantly larger than the previous constructions, which were on the order of $1/2^{n/2}$. We generalize their method and explicitly construct a large family of sets $A$ with $|A+A+A+A| > |(A+A)-(A+A)|$. The additional sums and differences allow us greater freedom than in \cite{MOS}, and we find that for any $\epsilon>0$ the density of such sets is at least $C/n^\epsilon$. In the course of constructing such sets we find that for any integer $k$ there is an $A$ such that $|A+A+A+A| - |A+A-A-A| = k$, and show that the minimum span of such a set is 30.
\end{abstract}

\maketitle

\tableofcontents

\section{Introduction}

Many problems in number theory reduce to understanding the behavior of sums and differences of a set with itself, where for a set $A$ the sumset is \be A + A \ = \ \{x+y: x, y \in A\}\ee and the difference set is \be A - A \ = \ \{x-y: x, y \in A\};\ee if $A$ is finite we denote the number of elements of $A$ is denoted by $|A|$. Examples include the Goldbach problem (if $\mathcal{P}$ is the set of all primes, then $\mathcal{P}+\mathcal{P}$ contains all the even numbers), Waring's problem, the Twin Prime Conjecture (there are infinitely many representations of 2 in $\mathcal{P}-\mathcal{P}$), and Fermat's Last Theorem. In studying these additive problems, it is natural to compare $|A+A|$ and $|A-A|$. If the sumset is larger, we say $A$ is sum-dominant, or a More Sums Than Differences (MSTD) set. While such sets were known to exist (see \cite{He,Ma,Na1,Na2,Na3,Ru1,Ru2,Ru3}) it was thought that they were rare. Specifically, it was believed that as $n\to\infty$ the percentage of subsets of $\{0,\dots,n-1\}$ that were sum-dominant tends to zero. Martin and O'Bryant \cite{MO} recently disproved this, showing that a positive percentage of sets are sum-dominant. They showed the percentage is at least $2 \cdot 10^{-7}$, which was improved by Zhao \cite{Zh2} to at least $4.28 \cdot 10^{-4}$ (Monte Carlo simulations suggest that approximately $4.5 \cdot 10^{-4}$ percent are sum-dominant). See \cite{ILMZ1} for a survey of the field, where these and other results (such as those in \cite{HM1,HM2}, which deal with varying the probability measure on $\{0,\dots,n-1\}$) are given.

It is natural to ask whether or not there is an explicit construction of large families of MSTD sets. Unfortunately, the proofs in \cite{MO,Zh2} are probabilistic, and do not lend themselves to a clean enumeration of such sets. Miller, Orosz and Scheinerman \cite{MOS} gave an explicit construction of MSTD sets (shortly thereafter, Zhao \cite{Zh1} gave a new method as well). Previously the largest explicit families had, in the limit, at least $f(n) / 2^{n/2}$ of the $2^n$ subsets of $\{0,\dots,n-1\}$ being sum-dominant (with $f$ a nice polynomial). The construction in \cite{MOS} gives an explicit family of size at least $C_4 / n^4$, which was improved to $C_1/n$ in \cite{Zh1}. The purpose of this paper is to extend the method in \cite{MOS} to generalized MSTD sets. While our families will not be a positive percentage, we see in Theorem \ref{thm:mainepsilon} that we can preserve the simplicity of the construction but improve the result to missing by an arbitrarily small power.

Before explaining Miller, Orosz and Scheinerman's construction, we first set some notation.

\bi

\item We let $[a,b]$ denote all integers from $a$ to $b$; thus $[a,b] = \{n \in \Z: a \le n \le b\}$.\\

\item By $kA$ we mean $A$ added to itself $k-1$ times: \be kA \ = \ \underbrace{A+\cdots+A}_{\mbox{k times}}. \ee

\item We say a set of integers $A$ has the property $P_n$ (or is a $P_n$-set) if both its sumset and its difference set contain all but the first and last $n$ possible elements (and of course it may or may not contain some of these fringe elements).\footnote{It is not hard to show that for fixed $0<\alpha\le1$ a random set drawn from $[1,n]$ in the uniform model is a $P_{\lfloor \alpha n\rfloor}$-set with probability approaching $1$ as $n\to\infty$.\label{footnote:beingpn}} Explicitly, let $a=\min{A}$ and $b=\max{A}$. Then $A$ is a $P_n$-set if \bea\label{eq:beingPnsetsum} [2a+n,\ 2b-n]  \ \subset\  A+A \eea and \bea\label{eq:beingPnsetdiff} [-(b-a)+n,\ (b-a)-n]\ \subset\ A-A.\eea \ \\ \

\ei


Essentially, their method is as follows (see \cite{MOS} for the full details). Let $A$ be an MSTD set, and write $A$ as $L \cup R$, where $L$ is the left fringe and $R$ the right fringe; for convenience, we assume $|L|=|R|=n = |A|/2$ and $1,2n\in A$. Let $O_k = [1,k] = \{1,\dots,k\}$, and for any $M$ of length $m$ set \bea A(M) & \ = \ & L \cup (n+O_k) \cup (n+k+M) \cup (n+k+m+O_k) \cup (n+ 2k+m+R), \eea where $a+S$ is the translate of $S$ by $a$. If $A$ is a $P_n$, MSTD set, then $A(M)$ is an MSTD set, so long as $M$ contains one out of every $k$ consecutive elements. The reason this is true is due to the fact that we have two intervals containing $k$ consecutive elements, and these guarantee that all possible sums are realized as $M$ never misses $k$ consecutive elements. This controls the middle of $A(M)+A(M)$; the fringes are controlled by $L$ and the translate of $R$. One way to ensure $M$ never misses $k$ consecutive elements is to divide $M$ into $m/(k/2)$ consecutive blocks of size $k/2$ (assume $\frac{k}{2}|m$), and note that each block may be any non-empty subset of a translate of $[1,k/2]$. The number of valid choices for all the blocks is \be \left(2^{k/2}-1\right)^{m/(k/2)}\ =\ 2^{m}\left(1 - \frac{1}{2^{k/2}}\right)^{m/(k/2)}; \ee the factor hitting $2^m$ measures how much we lose from our condition. There is also a loss from having two translates of $O_k$; we could have had $2^{2k}$ possible sets here, but instead have a fixed choice. Letting $r = 2n+2k+m$ and optimizing $m$ and $k$, Miller, Orosz and Scheinerman show their family has density at least $C/r^4$ for some $C>0$.

An essential ingredient in \cite{MOS} is the existence of one $P_n$, MSTD set $A$. While it is not hard to find such a set by brute force enumeration, this becomes tricker for the generalized problems we now consider. Instead of looking at $|A+A|$ versus $|A-A|$, one could study $|A+A+A|$ and $|A+A-A|$ or $|A+A+A+A|$ and $|A+A-A-A|$. While the methods of \cite{MOS} generalize to these (and additional) cases, the increased number of additions and subtractions provide opportunities that were not present in $A+A$ and $A-A$, and significantly larger families can be explicitly constructed once an initial set is found. For definiteness in this paper we mostly study sets with $|A+A+A+A| > |A+A-A-A|$, and we give an example where this holds. For general comparisons, Iyer, Lazarev, Miller and Zhang recently proved existence and positive percentage (see \cite{ILMZ1, ILMZ2} for the construction). Our main result is the following.

\begin{thm}\label{thm:mainepsilon}
For all $\epsilon>0$, there is a constant $C_\epsilon>0$ such that as $r$ goes to infinity, the percentage of subsets $A$  of $[1,r]$ with $|2A+2A| > |2A-2A|$ is at least $C_\epsilon/r^\epsilon$.
\end{thm}

\begin{rek} It is worth noting that Theorem \ref{thm:mainepsilon} gives us a higher percentage family of generalized MSTD sets (with $|2A+2A| > |2A-2A|$) than MSTD sets. Our methods generalize to $|4A+4A| > |4A-4A|$ (among other comparisons). \end{rek}


In the course of proving Theorem \ref{thm:mainepsilon}, our tools immediately yield

\begin{thm}\label{thm:maindiffisx}
Given $x\in \mathbb{Z}$ there exists an $S_x$ with $|2S_x+2S_x|-|2S_x-2S_x|=x$.
\end{thm}

In other words, we can construct these generalized MSTD sets such that we have arbitrarily more sums than differences.

In \S\ref{sec:constructingI} we go through (in full detail) the calculation needed to generalize \cite{MOS}, and obtain a lower bound for the probability of $C/r^{4/3}$. We improve this to $C'/r^\epsilon$ for any $\epsilon>0$ in \S\ref{sec:constructingII}, and then end in \S\ref{sec:givendiffs} by showing we can find sets such that the size of the generalized sumset is any desired number greater than (or less than) the generalized difference set. Not surprisingly, as the bounds for the density of these generalized MSTD sets improve, our constructions become more complicated; this is why we provide full details and a description of the method for the weaker results.

\section{Constructing many $A$ with $|2A+2A|>|2A-2A|$, I}\label{sec:constructingI}

In this section we generalize the construction in \cite{MOS}; we greatly improve the percentage in the next section. Here we prove

\begin{thm}\label{thm:firstgenMOS}
There is a constant $C>0$ such that as $r$ goes to infinity, the percentage of subsets of $[1,r]$ with $|2A+2A| > |2A-2A|$ is at least $C/r^{4/3}$.
\end{thm}

We first describe our search for one set with the desired properties (as our approach may be of use in finding sets needed for other problems), then discuss some lemmas needed to generalize Miller, Orosz and Scheinerman's construction.

We started by searching for a single set with $|A+A+A+A |>|A+A-A-A|$; from now on we use the notation $4A$ to denote $A+A+A+A$ and $2A-2A$ to denote $A+A-A-A$. We generated random subsets of $[1,40]$, including each number with probability $1/4$, and checked if the generated sets had our desired property. We quickly found  $A = \{6, 7, 9, 10, 13, 32, 35, 36, 38, 39, 40\}$, which has $|A+A+A+A|=136$ and $|A+A-A-A|=135$.


In order to construct an infinite family from one set $A$ using the techniques of \cite{MOS}, $A$ must satisfy two properties:

\begin{itemize}

\item The set $A$ must be a subset of $[1,2n]$ containing $1$ and $2n$.

\item The set $A$ must be a $P^4_n$ set; meaning that $4A$ and $2A-2A$ contain at least all but the first and last $n$ possible elements.

\end{itemize}

While we can subtract 5 from each element in our set, to have it start at 1 without affecting the number of sums and differences, it then ranges from 1 to 35 and 35 is not even. Though we could restructure their construction to avoid needing the first condition, our set does not meet the second condition either.
We then looked for further ways to modify our set, hoping to find a set that had $|4A |>|2A-2A|$ and meet their second condition. By taking our set and adding it to $\{0,49\}$ (that is repeating each element shifted by 49)\footnote{While it is expected that $A + \{0,a\}, a>4(34)=136$ would still have $|4A|>|2A-2A|$ (since the two repetitions of A would never interact) it is surprising that $A + \{0,a\}$ still has $|4A |>|2A-2A|$ for many smaller values of $a$. Investigating this might lead to some insight into the structure of sets with $|4A|>|2A-2A|$.}, we found our desired set. With $n=42$,
\be A\ = \ \{1, 2, 4, 5, 8, 27, 30, 31, 33, 34, 35,50,51,53,54,57,76,79,80,82,83,84\}\ \subset\ [1,2n], \ee $1, 2n \in A$, and
\be 4A\ = \ [4,336] \backslash \{27\} \supset [n+4, 7n], \ \ \ 2A-2A \ =\ [-166,166] \backslash \{141,-141\} \supset [-3n+1,3n]. \ee
This set thus meets all of the required properties to use a modified version of Miller et. al's construction of an infinite family of sets with $|4A |>|2A-2A|$. To do so, we first need to prove two lemmas, similar to their Lemma 2.1 and Lemma 2.2.

\begin{lem}\label{lem:lemma1} Let $A = L \cup R$ be a $P^4_n$ set where $L \subset [1,n]$ and $ R\subset [n+1,2n]$. Form $A' = L\cup M \cup R' $ where $M\subset [n+1,n+m]$ and $R' = R+m$. If $A'$ is a $P^4_n$ set then $|4A'|-|4A|=|2A'-2A' |-|2A-2A|=4m$ and thus if $|4A|>|2A-2A|$ then $|4A'|>|2A'-2A'|$. \end{lem}

The utility of this lemma is that if $A$ were also a generalized MSTD set (with $|4A| > |2A-2A|$), then $A'$ would be a generalized MSTD set as well.

\begin{proof}
We first consider the number of added sums. Just as in \cite{MOS}, in the interval $[4,n+3]$, $4A$ and $4A'$ are identical as all elements come from $L+L+L+L$. Also, we can pair the elements of $4A$ in the interval $[7n+1,8n]$ with the elements of $4A'$ in the interval $[7n+1+4m,8n+4m]$. Since both $A$ and $A'$ are $P^4_n$ sets, we know they each contain all possible elements more than $n$ from their boundaries. Having accounted for the sums within $n$ of the boundaries, $|4A'|-|4A| = (7n+4m+1)-(7n+1)=4m$.

Now consider the differences in the same way. Again, the elements within $n$ of the boundaries of $2A-2A$ and $2A'-2A'$ can be paired and both contain all elements that are not within $n$ of the boundaries (since they are  $P^4_n$ sets). The filled middle interval in $2A-2A$ is $[-3n+2, 3n-2]$ and in $2A'-2A'$ is $[-3n-2m+2, 3n+2m-2]$. Thus $|2A'-2A'|-|2A-2A|=4m$ as desired.
\end{proof}

\begin{lem}\label{lem:lemma2} Let $A = L \cup R$ be a $P^4_n$ set where $L \subset [1,n]$ and $R\subset [n+1,2n]$ and $\{1,2n\} \in A$. Form $A' = L \cup O_1 \cup M \cup O_2 \cup R $ with $O_1 = [n+1,n+k]$, $M\subset[n+k+1,n+k+m]$,
$O_2=[n+k+m+1,n+2k+m]$ and $R'=R+2k+m$. If $k\ge n$ and $M$ has no run of $3k-2$ missing elements then $A'$ is a $P_n^4$ set.
\end{lem}

\begin{proof}
We need to show that $4A' \supset [n+4, 7n+8k+4m] $ and $2A' - 2A' \supset [-3n-4k-2m+2,3n+4k+2m-2]$ (because $4A' \subset [4, 8n+8k+4m]$ and $2A'-2A' \subset [-4n-4k-2m+2,4n+4k+2m-2]$).

First consider $4A'$. Since $1 \in L$,
\bea L+L+L+O_1 &\ \supset\ & [4+n, 3+n+k]\nonumber \\
L+L+O_1+O_1 &\supset& [4+2n, 2+2n+2k] \nonumber\\
L+O_1+O_1+O_1& \supset& [4+3n, 1+3n+3k]. \eea

Further, since $2n \in R$,
\bea O_2+O_2+O_2+R'&\ \supset\ & [5n+5k+4m+3,5n+8k+4m]\nonumber \\
O_2+O_2+R'+R'& \supset & [6n+6k+4m+2,6n+8k+4m]\nonumber \\
O_2+R'+R'+R' &\supset & [7n+7k+4m+1,7n+8k+4m]. \eea

We now consider the sums of the $O_i$'s. We have \bea
O_1+O_1+O_1+O_1 & \ \supset\ & [4+4n, 4n+4k]\nonumber \\
O_1+O_1+O_1+O_2 &\supset & [4+4n+k+m, 4n+5k+m]\nonumber \\
O_1+O_1+O_2+O_2 &\supset &[4+4n+2k+2m, 4n+6k+2m]\nonumber \\
O_1+O_2+O_2+O_2 &\supset &[4+4n+3k+3m, 4n+7k+3m]\nonumber \\
O_2+O_2+O_2+O_2 &\supset & [4+4n+4k+4m, 4n+8k+4m].\eea

Finally, we study the sums involving $M$. We find
\be O_1+O_1+O_1+M\ =\ (O_1+O_1+O_1)+M\ =\ [3n+3,3n+3k] + M \supset [4n + 3k+1, 4n+k+m+3].\ee
This is because the smallest element in $M$ must be at most $n+3k-2$ and the largest element in $M$ is at least $m+n-2k+3$ (setting the bounds) and, because $M$ has no runs of $3k-2$ missing elements and $3O_1$ has $3k-2$ consecutive elements (closing the gaps). Similarly, \bea
O_1+O_1+O_2+M&\ \supset\ & [4n + 4k+m+1, 4n+2k+2m+3]\nonumber \\
O_1+O_2+O_2+M& \supset& [4n + 5k+2m+1, 4n+3k+3m+3]\nonumber \\
O_2+O_2+O_2+M& \supset& [4n + 6k+3m+1, 4n+4k+4m+3].\eea

Assembling these sums in the following order, and noting that the sums are contiguous, we get our desired result. \bea
L+L+L+O_1 & \ \supset \ & [4+n, 3+n+k]\nonumber\\
L+L+O_1+O_1 & \ \supset \ & [4+2n, 2+2n+2k]\nonumber\\
L+O_1+O_1+O_1 & \ \supset \ & [4+3n, 1+3n+3k]\nonumber\\
O_1+O_1+O_1+O_1 & \supset & [4+4n, 4n+4k]\nonumber\\
O_1+O_1+O_1+M & \supset & [4n + 3k+1, 4n+k+m+3]\nonumber\\
O_1+O_1+O_1+O_2 & \ \supset \ & [4+4n+k+m, 4n+5k+m] \nonumber\\
O_1+O_1+O_2+M & \ \supset \ & [4n + 4k+m+1, 4n+2k+2m+3]\nonumber\\
O_1+O_1+O_2+O_2 & \ \supset \ & [4+4n+2k+2m, 4n+6k+2m]\nonumber\\
O_1+O_2+O_2+M & \ \supset \ & [4n + 5k+2m+1, 4n+3k+3m+3]\nonumber\\
O_1+O_2+O_2+O_2 & \ \supset \ & [4+4n+3k+3m, 4n+7k+3m]\nonumber\\
O_2+O_2+O_2+M & \ \supset \ & [4n + 6k+3m+1, 4n+4k+4m+3] \nonumber\\
O_2+O_2+O_2+O_2 & \ \supset \ & [4+4n+4k+4m, 4n+8k+4m]\nonumber\\
O_2+O_2+O_2+R' & \ \supset \ & [5n+5k+4m+3,5n+8k+4m]\nonumber\\
O_2+O_2+R'+R' & \ \supset \ & [6n+6k+4m+2,6n+8k+4m]\nonumber\\
O_2+R'+R'+R' & \ \supset \ & [7n+7k+4m+1,7n+8k+4m]. \eea
Therefore $4A' \supset [4+n,7n+8k+4m]$.

Now consider $2A'-2A'$. Assembling the following sums (using the same logic concerning M): \bea
L+L-R'-O_2 & \ \supset \ & [2-3n-4k-2m, 1-3n-3k-m] \nonumber\\
L+L-O_2-O_2 & \ \supset \ & [2-2n-4k-2m,-2n-2k-2m] \nonumber\\
L+O_1-O_2-O_2 & \ \supset \ & [2-n-4k-2m, -1-m-k-2m] \nonumber\\
O_1+O_1-O_2-O_2 & \supset & [2-4k-2m, -2-2m] \nonumber\\
M+O_1-O_2-O_2 & \ \supset \ & [-1-k-2m,1-3k-m]\nonumber\\
O_1+O_2-O_2-O_2 & \supset & [2-3k-m,k-m-2]\nonumber\\
M+O_2-O_2-O_2 & \ \supset \ & [-1-m,1-2k]\nonumber\\
O_2+O_2-O_2-O_2 & \ \supset \ & [2-2k,-2+2k].\eea

Since these regions are all contiguous, $2A'-2A' \supset [2-3n-4k-2m, 0]$. Since $2A'-2A'$ must be symmetric about 0, $2A'-2A' \supset [2-3n-4k-2m, -2+3n+4k+2m]$ as desired. Therefore $A'$ is a $P_n^4$-set.
\end{proof}

Using these lemmas, we can now prove Theorem \ref{thm:firstgenMOS}.

\begin{proof}[Proof of Theorem \ref{thm:firstgenMOS}]
Just as in the proof in \cite{MOS}, we need to count the number of sets $M$ of the form $O_1 \cup M \cup O_2$ of width $r=2k+m$ which may be inserted into a $P^4_n$-set $A$ with $|4A|>|2A-2A|$. We are counting the exact same sets as in \cite{MOS}, except for them there $M$ could not contain any run of $k$ consecutive elements whereas ours cannot contain any run of $3k-2$ missing elements. They could ensure their condition was satisfied by requiring each block of $k/2$ must contain at least one element; the analogous condition for us is that each block of size $\frac{3k}{2} -1$ must contain at least one element. We can ignore the minus 1, since it will not matter as $r$ gets large.

Following the same logic as in \cite{MOS}, we end up needing the asymptotic behavior of the sum \be \sum_{k=n}^{r/4} \frac{1}{2^{2k}} \left(1-\frac{1}{2^{3k/2}}\right)^\frac{r}{{3k}/2}.\ee Note the factors of $1/2^{2k}$ arise from taking sets $O_i$ that are $k$ consecutive elements, and the factor $(1-\frac{1}{2^{3k/2}})$ is due to our condition of $M$ having at least one element in blocks of size $3k/2$.

Fortunately, in anticipation of this work, \cite{MOS} analyzed the more general sum \be \sum_{k=n}^{r/4} \frac{1}{2^{ak}} \left(1-\frac{1}{2^{bk}}\right)^{r/ck},\ee showing there is a constant $C>0$ such that it is at least $C/r^{a/b}$ (see their Lemma 3.1). Our sum is of the same form with parameters $a=2$, $b=c=3/2$, and thus our sum is at least $1/r^{4/3}$.
\end{proof}

\begin{rek}\label{rek:schilling} Our density bound above is related to the bounds from \cite{MOS} for sum-dominant sets, and an improvement there translates to an improvement here. We describe a simple improvement one can make to the arguments in \cite{MOS}, which allows us to replace the $1/r^4$ they obtained for sum-dominant sets with a $1/r^2$. While a similar analysis would improve our results here, we choose not to do so as the real improvement comes from a better choice of the $O$'s (described in the next section) and not the middle.

We appeal to an analysis of the probability $m$ consecutive tosses of a fair coin has its longest streak of consecutive heads of length $\ell$ (see \cite{Sc}). While the expected value of $\ell$ grows like $\log_2(m/2)$, the variance converges to a quantity independent of $m$, implying an incredibly tight concentration. If we take $O_1$ and $O_2$ as before and of length $k$, we may take a positive percentage of all $M$'s of length $m$ to insert in the middle, so long as $k = \log_2(m/2) - c$ for some $c$. The size of $A$ is negligible; the set has length essentially $r = m+2k$. Of the $2^{m+2k}$ possible middles to insert, there are $C 2^m$ possibilities (we have a positive percentage of $M$ work, but the two $O$'s are completely forced upon us). This gives a percentage on the order of $2^m / 2^{m+2k}$; as $k=\log_2(m/2)-c$, this gives on the order of $1/r^2$ as a lower bound for the percentage of sum-dominated sets, much better than the previous $1/r^4$.
\end{rek}

\section{Constructing many $A$ with $|2A+2A|>|2A-2A|$, II}\label{sec:constructingII}

We discuss improvements to the exponent in Theorem \ref{thm:firstgenMOS}. The following two observations are very important in improving our exponent.

\begin{itemize}

 \item The $O$'s always show up at least in pairs in the sums and differences used to prove $A'$ was a $P_n^4$-set, except in cases where they show up with $L+L+L$, $R'+R'+R'$ or $L+L-R'$.
 \item Each of $L+L+L$, $R'+R'+R'$ and $L+L-R'$ contain a run of 16 elements in a row.

\end{itemize}

These two points allow us to relax our structure for each of the $O$'s and still have all of the sums and differences just stated fill the necessary ranges. This greatly improves our exponent, as we lost a power due to the $1/2^{2k}$ factor from the $O$'s. As long as each $O$ contains its first and last possible element, each $O$ has no run of 16 missing elements and $2O=O+O$ is full for both $O$'s, $A'$ will be a $P_n^4$-set. This looser structure allows us to replace the $1/2^{2k}$ with a much better factor and thus greatly improve our density bound.

\begin{thm}\label{thm:run16approxf}
There is a constant $C>0$ such that as $r$ goes to infinity, the percentage of subsets of $[1,r]$ with $|4A| > |2A-2A|$ is at least $C/n^r$, where $r = \frac16 \log_2(256/255) \approx .001$.
\end{thm}

\begin{proof}
Instead of requiring that $O_1$ and $O_2$ contain all elements in their ranges as before, we now only require that they contain the first $16$ elements, the last element, and no runs of $16$ consecutive non-chosen in between. While the old $O$'s contributed $1/2^{2k}$ to our sum, the new $O$'s contribute significantly more. For use later in proving Theorem \ref{thm:mainepsilon}, we analyze this problem more generally, and force each $O$ to contain the first $f$ elements, the last element, and no run of $f$ missed elements in between. We again use a crude bound to ensure that each $O$ contains no runs of $f$ blanks and force each $O$ to contain at least one element in every block of $f/2$ elements.

In each $O$ we thus have at most $2k/f$ blocks of length $f/2$. In each block, there are $2^{f/2}$ options and all but one contain at least one element. The fraction of subsets that work as $O$'s is thus at least \be \left(\frac{2^{f/2}-1}{2^{f/2}}\right)^{2k/f}. \ee
To use in our sum, we want to represent this expression as $2^{-\alpha k}$.  We find \bea
\left(\frac{2^{f/2}-1}{2^{f/2}}\right)^{2k/f} & \ = \ & 2^{- \alpha k}\nonumber\\
\frac{2k}{f}\log_2\frac{2^{f/2}-1}{2^{f/2}} &=& - \alpha k\nonumber\\
\alpha &=& \frac{-2}{f}\log_2\frac{2^{f/2}-1}{2^{f/2}}.
\eea
Since, for our current purposes, $f=16$, we find $\alpha=\frac{-1}{8}\log_2\frac{255}{256}$. We know that our sum guarantees a bound of $1/r^{a/b}$, we know from before that $b={3/2}$ and now know that $a=2\alpha \approx 0.00142$ (because there are 2 $O$'s).
Thus there exists some constant $C$ for which the percentage of subsets of $[1,n]$ is greater than $C/n^r$, where $r = \frac16 \log_2(256/255) \approx .001$
\end{proof}

This construction could be pushed further by finding a `better' $A$. If $L+L+L$, $R+R+R$ and $L+L-R$ contained longer runs, we would have more freedom in each $O$, and thus could form a better bound. Rather than look for more sets, however, to allow ourselves to push the bound even further, we modify our construction slightly, and add two more components to our $A'$.

Starting with $A$ as in Theorem \ref{thm:firstgenMOS} (a $P_4^n$ subset of $[1,2n]$ that contains $1$ and $2n$) we form \be A' \ = \ L \cup F_1 \cup O_1 \cup M \cup O_2 \cup F_2 \cup R', \ee where
\begin{itemize}

\item $L \subset[1,n]$ containing $1$,
\item $F_1 = [n+1,n+f]$,
\item $O_1 \subset [n+f+1,n+f+k]$ containing the first $f$ elements, the last element, and no runs of $f$ missing elements,
\item $M \subset [n+f+k+1, n+f+k+m]$ with no runs of $k$ missing elements,
\item $O_2 \subset [n+f+k+m+1, n+f+2k+m]$ containing the first $f$ elements, the last element, and no runs of $f$ blanks,
\item $F_2 = [n+f+2k+m+1, n+2f+2k+m]$,
\item $R' \subset [n+2f+2k+m+1, 2n+2f+2k+m]$ containing $2n+2f+2k+m$.

\end{itemize}

By a method similar to that used in Lemma \ref{lem:lemma2}, we can prove that these $A'$s are $P_4^n$-sets. Since $A$ has $|4A|>|2A-2A|$, we have $|4A'|>|2A'-2A'|$ (by Lemma \ref{lem:lemma1}).

With this new construction, we now prove our best lower bound for the density of sets with $|4A| > |2A-2A|$.

\begin{proof}[Proof of Theorem \ref{thm:mainepsilon}]
For any fixed $f$, we can form \be A'\ =\ L \cup F_1 \cup O_1 \cup M \cup O_2 \cup F_2 \cup R', \ee with $A$ as in the proof of Theorem \ref{thm:run16approxf}. To ensure there are no overly long runs of missing elements, we force each $O$ to contain one element in every block of $f/2$ and $M$ to contain one element in every block of $k/2$.\footnote{We could weaken this construction by appealing to results on the length of consecutive heads in tosses of a fair coin; see Remark \ref{rek:schilling}. As this will not change the form of our final bound, we prefer to keep the exposition simple.}

By allowing $k$ to grow, and summing over our possible sets as we did in Theorems \ref{thm:firstgenMOS} and \ref{thm:run16approxf}, we know that the proportion of subsets $[1,n]$ that have $|4A|>|2A-2A|$ is at least
\be\sum_{k=n}^{r/4} \frac{1}{2^{2\alpha}} \left(1-\frac{1}{2^{k/2}}\right)^\frac{r}{{k}/2}, \ \ \ \alpha\ =\ \frac{-2}{f}\log_2\frac{2^{f/2}-1}{2^{f/2}}. \ee
From Lemma 3.1 of \cite{MOS}, this sum is at least $C/r^p$, with \be p\ = \ \frac{2\alpha}{1/2}\ = \ 4\frac{-2}{f}\log_2\frac{2^{f/2}-1}{2^{f/2}}. \ee
Since \be \lim_{f\to\infty}4\frac{-2}{f}\log_2\frac{2^{f/2}-1}{2^{f/2}} \ =  \ 0, \ee we can force the bound to be better than $C/n^\epsilon$ for any $\epsilon > 0$, completing the proof.
\end{proof}

\section{Constructing Generalized MSTD sets with given differences}\label{sec:givendiffs}

We explore some consequences of our constructions. We first prove that given any $x$ there is an $A$ with $|4A|-|2A-2A| = x$.

\begin{proof}[Proof of Theorem \ref{thm:maindiffisx}]
We first consider negative $x$. Let \be S_x\ =\ [1,|x|+2]\cup\{2|x|+3\}.\ee Then \be 4S_x\ =\ [4, 7|x|+11]\cup\{8|x|+12\}, \ \ \ 2S_x-2S_x=[-4|x|-4,4|x|+4].\ee Thus $|4S_x|-|2S_x-2S_x|=(7|x|+9)-(8|x|+9)=-|x|=x$ as desired.\\

For $x=0$, let $S_0=\{0\}$. Then $|4S_0|-|2S_0-2S_0|=1-1=0$.\\

We are left with positive values of $x$. Similar to Martin and O'Bryant's \cite{MO} proof that $|A+A|-|A-A|$ can equal any value, we deal with certain small values of $x$ explicitly, then offer a method of construction for larger values of $x$. Let \be S_1 \ = \ \{0, 1, 3, 4, 7, 26, 29, 30, 32, 33, 34\}; |S_1+S_1+S_1+S_1|-|S_1+S_1-S_1-S_1| \ = \ 1. \ee
Now consider the positive values of $x\equiv 1 \bmod 4$, so $x=4k+1$. Define \be\label{eq:defnSfourkplusone} S_{4k+1} \ = \ S_1+\{0,137,274,\dots,137k\}. \ee
Then \bea 4S_{4k+1} & \ = \ & \{0 \leq s \leq 137(4k+1) -1: s \not\equiv 23 \bmod 137\} \nonumber\\
 2S_{4k+1}-2S_{4k+1} &= & \{-137(2k+1/2) < s < 137(2k+1/2) : s \not\equiv 43, 231 \}; \eea
Thus $|4S_{4k+1}| - |2S_{4k+1}-2S_{4k+1}| = ((4k+1)\cdot 136) - ((4k+1)\cdot 135) = 4k+1$ as desired.\\

Next we consider the positive values of $x \equiv 0 \bmod 4$. With $S_{4k+1}$ as in \eqref{eq:defnSfourkplusone}, define \be S_{4k}\ =\ S_{4k+1} \backslash \{137\}. \ee After some algebra we find that $2S_{4k}-2S_{4k} = 2S_{4k+1}-2S_{4k+1}$ but  $4S_{4k} = 4S_{4k+1} \backslash \{137\}$. Thus \be |4S_{4k}| - |2S_{4k}-2S_{4k}| \ =  \ |4S_{4k+1}| - |2S_{4k+1}-2S_{4k+1}|-1\ = \ 4k,\ee as desired.\\

Next, we study the positive values of $x \equiv 2 \bmod 4$. Again, with $S_{4k+1}$ as in \eqref{eq:defnSfourkplusone}, define \be S_{4k-6} \ = \ S_{4k+1} \backslash \{34\}. \ee We have \bea 4S_{4k-6} &\ = \ & 4 S_{4k+1} \backslash \{109, 133, 134, 135, 136, 246, 271, 272, 273, 383,409,410,547\}\nonumber\\
2S_{4k-6}-2S_{4k-6} &=& 2S_{4k+1}-2S_{4k+1} \backslash \{-481,-480,-343,343,480,481\}. \eea
Thus \be |4S_{4k-6}| - |2S_{4k-6}-S_{4k-6}|=|4S_{4k+1}| - |2S_{4k+1}-S_{4k+1}|+13-6 \ = \ 4k+1-7 \ = \ 4k-6, \ee as desired.\\

Finally, the take care of the positive values of $x \equiv 3 \bmod 4$. Here we set \be S_{4k-1} \ = \ S_{4k+1} \backslash \{33\}, \ee where as always $S_{4k+1}$ is as in \eqref{eq:defnSfourkplusone}. We have \bea 4S_{4k-1} &\ =\ & 4S_{4k+1} \backslash \{133, 135\} \nonumber\\ 2S_{4k-1}-2S_{4k-1} &=& 2S_{4k+1}-2S_{4k+1}. \eea
Thus \be |4S_{4k-1}| - |2S_{4k-1}-S_{4k-1}| \ =\ |4S_{4k+1}| - |2S_{4k+1}-S_{4k+1}|-2 \ = \ 4k-1, \ee completing the proof.
\end{proof}

\ \\

\begin{thm}\label{thm:minimumspan}
The minimum span for any set with $|4A|>|2A-2A|$ is 30.
\end{thm}

\begin{proof}
There are no subsets of $[1,30]$ with $|4A|>|2A-2A|$, which can be checked by brute force in a reasonable amount of time as $2^{29} < 10^9$. Thus the minimum span cannot be less than 29. As  $A = \{1, 2, 3, 5, 9, 24, 28, 30, 31\}$ has $|4A| > |2A-2A|$, the minimum span must be 30.
\end{proof}

\ \\

\begin{thebibliography}{KonS8}


\bibitem[FP]{FP}
G. A. Freiman and V. P. Pigarev, \emph{The relation between the invariants R and T}, Number theoretic studies in the Markov spectrum and in the structural theory of set addition (Russian),
Kalinin. Gos. Univ., Moscow, 1973, 172--174.

\bibitem[He]{He}
P. V. Hegarty, \emph{Some explicit constructions of sets with more sums than differences} (2007), Acta Arithmetica \textbf{130} (2007), no. 1, 61--77.

\bibitem[HM1]{HM1}
P. V. Hegarty and S. J. Miller, \emph{When almost all sets are difference dominated}, Random Structures and Algorithms \textbf{35} (2009), no. 1, 118--136. 

\bibitem[HM2]{HM2}
P. V. Hegarty and S. J. Miller, Appendix 2 of \emph{Explicit constructions of infinite families of MSTD sets} (by S. J. Miller and D. Scheinerman), Additive Number Theory: Festschrift In Honor of the Sixtieth Birthday of Melvyn B. Nathanson (David Chudnovsky and Gregory Chudnovsky, editors), Springer-Verlag, 2010.

\bibitem[ILMZ1]{ILMZ1}
G. Iyer, O. Lazarev, S. J. Miller and L. Zhang, \emph{Finding and Counting MSTD sets}, to appear in the conference proceedings of the 2011 Combinatorial and Additive Number Theory Conference, \texttt{http://arxiv.org/abs/1107.2719}.

\bibitem[ILMZ2]{ILMZ2}
G. Iyer, O. Lazarev, S. J. Miller and L. Zhang, \emph{Generalized More Sums Than Differences Sets}, to appear in the Journal of Number Theory, doi:10.1016/j.jnt.2011.10.006.

\bibitem[Ma]{Ma}
J. Marica, \emph{On a conjecture of Conway}, Canad. Math. Bull. \textbf{12} (1969), 233--234.

\bibitem[MO]{MO}
G. Martin and K. O'Bryant, \emph{Many sets have more sums than differences}, in Additive Combinatorics, CRM Proc. Lecture Notes, vol. 43, Amer. Math. Soc., Providence, RI, 2007, pp. 287--305.

\bibitem[MOS]{MOS}
S. J. Miller, B. Orosz and D. Scheinerman, \emph{Explicit constructions of infinite families of MSTD sets}, Journal of Number Theory \textbf{130} (2010) 1221--1233.


\bibitem[Na1]{Na1}
M. Nathanson, \emph{Additive Number Theory: The Classical Bases},
Graduate Texts in Mathematics, Springer-Verlag, New York, $1996$.

\bibitem[Na2]{Na2}
M. B. Nathanson, \emph{Problems in additive number theory, 1},
Additive combinatorics, 263--270, CRM Proc. Lecture Notes \textbf{43},
Amer. Math. Soc., Providence, RI, 2007.

\bibitem[Na3]{Na3}
M. B. Nathanson, \emph{Sets with more sums than differences},
Integers : Electronic Journal of Combinatorial Number Theory \textbf{7} (2007), Paper A5 (24pp).

\bibitem[Ru1]{Ru1}
I. Z. Ruzsa, \emph{On the cardinality of $A + A$ and $A - A$}, Combinatorics year (Keszthely, 1976), vol. 18, Coll. Math. Soc. J. Bolyai, North-Holland-Bolyai T$\grave{{\rm a}}$rsulat, 1978, 933--938.

\bibitem[Ru2]{Ru2}
I. Z. Ruzsa, \emph{Sets of sums and differences}, S$\acute{{\rm e}}$minaire de Th$\acute{{\rm e}}$orie des Nombres de Paris 1982-1983 (Boston), Birkh$\ddot{{\r a}}$user, 1984, 267--273.

\bibitem[Ru3]{Ru3}
I. Z. Ruzsa, \emph{On the number of sums and differences}, Acta Math. Sci. Hungar. \textbf{59} (1992), 439--447.

\bibitem[Sc]{Sc}
M. F. Schilling, \emph{The longest run of heads}, The College Mathematics Journal \textbf{21} (1990), no. 3, 196--207.




\bibitem[Zh1]{Zh1}
Y. Zhao, \emph{Constructing MSTD Sets Using Bidirectional Ballot Sequences}, Journal of Number Theory \textbf{130} (2010), no. 5, 1212--1220.

\bibitem[Zh2]{Zh2}
Y. Zhao, \emph{Sets Characterized by Missing Sums and Differences}, Journal of Number Theory \textbf{131} (2011), 2107--2134. 

\end{thebibliography}
\end{document}